\documentclass[10pt]{amsart}
\usepackage{epsfig}
\usepackage{graphicx}
\usepackage{amscd}
\usepackage{amsmath}
\usepackage{amsxtra}
\usepackage{amsfonts}
\usepackage{amssymb}

\newtheorem{theorem}{Theorem}[section]
\newtheorem{corollary}[theorem]{Corollary}
\newtheorem{lemma}[theorem]{Lemma}
\newtheorem{proposition}[theorem]{Proposition}

\theoremstyle{definition}
\newtheorem{definition}[theorem]{Definition}
\newtheorem{remark}[theorem]{Remark}

\theoremstyle{remark}

\renewcommand{\theclaim}{\textup{\theclaim}}

\numberwithin{equation}{section}

\def\openone

{\mathchoice

{\hbox{\upshape \small1\kern-3.3pt\normalsize1}}

{\hbox{\upshape \small1\kern-3.3pt\normalsize1}}

{\hbox{\upshape \tiny1\kern-2.3pt\SMALL1}}

{\hbox{\upshape \Tiny1\kern-2pt\tiny1}}}

\makeatletter

\newbox\ipbox

\newcommand{\diracb}[1]{\left\langle #1\mathrel{\mathchoice

{\setbox\ipbox=\hbox{$\displaystyle \left\langle\mathstrut
#1\right.$}

\vrule height\ht\ipbox width0.25pt depth\dp\ipbox}

{\setbox\ipbox=\hbox{$\textstyle \left\langle\mathstrut
#1\right.$}

\vrule height\ht\ipbox width0.25pt depth\dp\ipbox}

{\setbox\ipbox=\hbox{$\scriptstyle \left\langle\mathstrut
#1\right.$}

\vrule height\ht\ipbox width0.25pt depth\dp\ipbox}

{\setbox\ipbox=\hbox{$\scriptscriptstyle \left\langle\mathstrut
#1\right.$}

\vrule height\ht\ipbox width0.25pt depth\dp\ipbox}

}\right. }

\newcommand{\dirack}[1]{\left. \mathrel{\mathchoice

{\setbox\ipbox=\hbox{$\displaystyle \left.\mathstrut
#1\right\rangle$}

\vrule height\ht\ipbox width0.25pt depth\dp\ipbox}

{\setbox\ipbox=\hbox{$\textstyle \left.\mathstrut
#1\right\rangle$}

\vrule height\ht\ipbox width0.25pt depth\dp\ipbox}

{\setbox\ipbox=\hbox{$\scriptstyle \left.\mathstrut
#1\right\rangle$}

\vrule height\ht\ipbox width0.25pt depth\dp\ipbox}

{\setbox\ipbox=\hbox{$\scriptscriptstyle \left.\mathstrut
#1\right\rangle$}

\vrule height\ht\ipbox width0.25pt depth\dp\ipbox}

} #1\right\rangle}

\newcommand{\diam}{\operatorname*{diam}}

\newcommand{\cj}[1]{\overline{#1}}

\newcommand{\bz}{\mathbb{Z}}

\newcommand{\B}{\mathcal{B}}
\newcommand{\br}{\mathbb{R}}

\newcommand{\bn}{\mathbb{N}}

\def\blfootnote{\xdef\@thefnmark{}\@footnotetext}

\input xy
\xyoption{all}
\usepackage{amssymb}

\newcommand{\supp}[1]{\text{supp} (#1)}

\def\F{\mathcal{F}}

\def\E{\mathcal E}

\def\-{^{-1}}
\def\B{\mathcal{B}}
\def\D{\mathcal{D}}

\begin{document}

\title[Continuous and Discrete Fourier Frames for Fractal Measures]{Continuous and Discrete Fourier Frames for Fractal Measures}
\author{Dorin Ervin Dutkay}
\blfootnote{This research is supported in part by  the National
Science Foundation grant 1106934}
\address{[Dorin Ervin Dutkay] University of Central Florida\\
    Department of Mathematics\\
    4000 Central Florida Blvd.\\
    P.O. Box 161364\\
    Orlando, FL 32816-1364\\
U.S.A.\\} \email{Dorin.Dutkay@ucf.edu}

\author{Deguang Han}
\address{[Deguang Han]University of Central Florida\\
    Department of Mathematics\\
    4000 Central Florida Blvd.\\
    P.O. Box 161364\\
    Orlando, FL 32816-1364\\
U.S.A.\\} \email{deguang.han@ucf.edu}

\author{Eric Weber}
\address{[Eric Weber]Department of Mathematics\\
396 Carver Hall\\
Iowa State University\\
Ames, IA 50011\\
U.S.A.\\} \email{esweber@iastate.edu}

\thanks{}

\begin{abstract}
Motivated by the existence problem of Fourier frames on fractal
measures, we introduce Bessel and frame measures for a given
finite measure on $\br^d$, as extensions of the notions of Bessel
and frame spectra that correspond to bases of exponential
functions. Not every finite compactly supported Borel measure
admits frame measures. We present a general way of constructing Bessel/frame measures for a given measure. The idea is that if a convolution of two measures admits a Bessel measure then one can use the Fourier transform of one of the measures in the convolution as a weight for the Bessel measure to obtain a Bessel measure for the other measure in the convolution. The same is true for frame measures, but with certain restrictions.
We investigate some general properties of frame measures and their Beurling dimensions. In particular we show that the
 Beurling dimension is invariant under convolution (with a
 probability measure) and under a certain type of discretization.
 Moreover, if a measure admits a frame measure then it admits an atomic one, and hence a weighted Fourier frame. We
also construct some examples of frame measures for self-similar
measures.
\end{abstract}
\maketitle \tableofcontents

\section{Introduction}

A fundamental result of Fourier analysis is the Plancherel
theorem:  for Lebesgue measure $\lambda$ on $\br$,
\begin{equation} \label{E:Plancherel}
\int_{\br} | \hat{f}(t)|^2 d \lambda (t) = \int_{\br}| f(x) |^2 d
\lambda (x).
\end{equation}
This equality suggests the idea of $\lambda$ as a dual measure to
itself, in that the norm of the Fourier transform of $f$ is equal
to the norm of $f$.  Similarly, for the measure $m = \lambda
|_{[0,1]}$, $\lambda$ is a dual measure in the same sense, since
if $f$ is supported on $[0,1]$, then equation (\ref{E:Plancherel})
holds.  Yet there is another measure that satisfies this norm
equivalence, namely the measure $\nu = \sum_{n \in \bz}
\delta_{n}$ since
\[ \int_{\br} | \hat{f}(t) |^2 d \nu(t) = \sum_{n \in \bz} | \int_{\br} f(x) e^{-2 \pi i n x} dm(x) |^2 \]
which equals $ \int_{\br} |f(x)|^2 d m (x) $ by the virtue that the
integer exponentials form an orthonormal basis for $L^2[0,1]$.
Moreover, for any sequence $\{ e^{2 \pi i \gamma_{n} x } : n \in
\bz \}$ which forms a Fourier frame \cite{DS52a,OSANN} for
the Paley-Weiner space, $\nu = \sum_{n \in \bz}
\delta_{\gamma_{n}}$ will be a dual measure in the slightly more
general sense that $\|f\|_{L^2(m)} \simeq \|\hat{f}\|_{L^2(\nu)}$.
In this paper, we consider questions of existence and structure of
dual measures for singular measures, in particular measures which
are invariant under iterated function systems.

In general, we consider a Borel measure $\nu$ on $\br^d$ to
be dual to $\mu$ if for every $f \in L^2(\mu)$,
\begin{equation} \label{E:equiv}
 \int_{\br^d} | \widehat{f\,d\mu} (x) |^2 d \nu(x) \simeq \int_{\br^d} | f(t) |^2 d \mu(t),
\end{equation}
where the Fourier transform is given by
$$\widehat{f\,d\mu}(t)=\int f(t)e^{-2\pi i t\cdot x}\,d\mu(x),\quad(t\in\br^d).$$
for a function $f\in L^1(\mu)$.  As we shall see, this is a generalization of the idea of a
\emph{Fourier frame}.

This idea of dual measures is closely related to other concepts.
Jorgensen and Pedersen consider \emph{spectral pairs} in
\cite{MR1700084}.  These spectral pairs consist of two measures
$\mu$ and $\nu$ on $\br^d$ such that the equivalence in Equation
(\ref{E:equiv}) is an equality
together with the requirement that the Fourier transform from
$L^2(\mu)$ to $L^2(\nu)$ is onto.  They consider specifically
compactly supported $\mu$ and purely atomic $\nu$.  Similarly,
problems of equivalent norms in Paley-Wiener spaces are considered
in \cite{OSANN,LS02a}.

In a slightly different view, the series of papers by Strichartz
\cite{MR1078738,MR1081941,MR1237052} considers a type of
Plancherel duality between self-similar measures and a limit of
localized Lebesgue measures.  As a typical result, Strichartz
proves that for a suitable fractal measure $\mu$ on $\br^d$ which
is ``locally $\alpha$-dimensional''
\[ \int_{\br^d} |f(x)|^2 d \mu \simeq \limsup_{R \to \infty} \frac1{R^{d - \alpha}} \int_{B_{0}(R)} | \widehat{f\,d\mu}(t)|^2 d \lambda. \]

Extensions of these results are contained in \cite{MZ09a}.

Finally, \cite{GH03a} introduce a concept of continuous frame,
that is a frame which is not a sequence of vectors in a Hilbert
space but instead a set of vectors parametrized by a measure
space.  The dual measure considered in the present paper is a
concrete case of Gabardo and Han's definition, where here the
vectors are specifically exponential functions.

In recent years there has been a wide range of interests in
expanding the classical Fourier analysis to fractal or more
general probability measures \cite{DHS09,HL08,
JP98,DHS09,HL08,MR1744572,JP98,MR2338387,MR2200934,MR2297038,MR1785282,MR2279556,MR2443273}.
One of the central themes of this area of research involves
constructive and computational bases in $L^2(\mu)$, where $\mu$ is
a measure which is determined by some self-similarity property.
These include classical Fourier bases, as well as wavelet and
frame constructions.

We are motivated by questions of Fourier frames for fractal
measures \cite{DHSW10,DHW11a}.  Specifically, we are motivated by
the question of whether the Cantor measure has a Fourier frame.
For the  middle-third Cantor set, there is a canonical measure
$\mu_{3}$ which is supported on the Cantor set--it is known that
there is no orthonormal basis of exponentials for $L^2(\mu_{3})$
\cite{JP98}.  It was shown in \cite{DHW11a} that there exists a
sequence of exponentials which has (relatively) large
Beurling dimension \cite{CKS08} and forms a Bessel sequence.  It
is still unknown if there is a Fourier frame for $\mu_{3}$.

In contrast, a Cantor like set with an associated measure
$\mu_{4}$ constructed in \cite{JP98} does possess an orthonormal
basis of exponentials.  It was shown in \cite{DHW11a} that any
Fourier frame for $\mu_{4}$ must have the property that the
sequence of frequencies must have Beurling dimension at most
$\dfrac{1}{2}$.  However, we show in Corollary \ref{cor3.5}, that
the integer lattice can be weighted so that a sequence of weighted
exponentials forms a frame for $\mu_{4}$, but the sequence of
integers which have a non-zero weight has Beurling dimension 1. In
the context of dual measures, there exists a sequence of weights
$\{d_n\}$ such that $\nu = \sum_{n \in \bz} d_n \delta_{n}$ is a
dual measure to $\mu_{4}$. We will introduce an appropriate notion
of Beurling dimension for measures and this measure $\nu$ will
have dimension $\dfrac{1}{2}$ as predicted by Theorem \ref{th2.6}.

The rest of this paper is organized as follows: we introduce basic definitions next. In section 2 we
investigate some general properties about Bessel and frame
measures. We demonstrate in  Theorem \ref{th1.1.1} that  there
cannot be any general statement concerning the existence of frame
measures. Moreover, we cannot even say that frame measures exist
for invariant measures of iterated function systems, which is our
motivating example.  However, by using convolutions of measures,
we show that if one frame measure exists, then in general many
frame measures exist (Proposition \ref{pr1.2}).  Moreover, if one
frame measure exists, then there exists a frame measure which is
absolutely continuous with respect to Lebesgue measure and has
smooth Radon-Nikodym derivative (Corollary \ref{cor1.3}).  At the
other extreme, again if one frame measure exists, then there
exists a frame measure which is atomic (Theorem \ref{th1.10}).  As
a consequence, if a measure, such as $\mu_{3}$ has a frame
measure, then it has a weighted Fourier frame (Remark
\ref{rm1.1}). Section 3 is devoted to establishing the connections
between frame/Bessel measures and  Beurling dimension and density.
We prove that the Beurling dimension is invariant under
convolutions with probability measures, and under discretizations
(Theorem \ref{th2.4} and Theorem \ref{th2.5}). Moreover, we obtain
that any Bessel measure with positive lower Beurling density is
absolutely continuous with respect to the Lebesque measure, and
its Radon-Nikodym derivative in $L^{2}$-integrable. The last
section of this paper is focused on frame measures for self-affine
measures. In Theorem \ref{th3.4} we prove that under certain tiling assumptions one can use the Fourier transform of the complementing measure as a weight for the Lebesgue measure and obtain a Plancherel measure for the given fractal measure.


\begin{definition}
Denote by $\delta_a$ the Dirac measure at the point $a$.  Denote by $e_t$, $t\in\br^d$, the exponential function
$$e_t(x)=e^{2\pi i t\cdot x},\quad(x\in\br^d).$$
\end{definition}

\begin{definition}
A sequence $\{x_n\}_{n=1}^{\infty}$ in a Hilbert space (with inner
product $\langle \cdot , \cdot \rangle $) is \emph{Bessel} if
there exists a positive constant $B$ such that
\[ \sum_{n=1}^{\infty} | \langle v , x_n \rangle |^2 \leq B \|v\|^2 \mbox{ for all }v.\]

The sequence is a frame if in addition to being a Bessel sequence
there exists a positive constant $A$ such that
\[ A \| v\|^2 \leq \sum_{n=1}^{\infty} | \langle v , x_n \rangle |^2 \leq B \|v\|^2. \]
In this case, $A$ and $B$ are called the lower and upper frame
bounds, respectively.

\end{definition}

We extend these ideas as follows.
\begin{definition}
We say that a Borel measure $\nu$ is a {\it Bessel measure} for $\mu$ if there exists a constant $B>0$ such that for every $f \in L^2(\mu)$, we have
\[ \|\widehat{f\,d\mu} \|_{L^2(\nu)}^{2}\leq B \| f\|_{L^2(\mu)}^{2} . \]
We call $B$ a {\it (Bessel) bound} for $\nu$. We say the measure
$\nu$ is a {\it frame measure} for $\mu$ if there exists constants
$A,B > 0$ such that for every $f \in L^2(\mu)$, we have
\[ A\| f\|_{L^2(\mu)}^{2}\leq \|\widehat{f\,d\mu} \|_{L^2(\nu)}^{2} \leq  B \| f\|_{L^2(\mu)}^{2}. \]
We call $A,B$ {\it (frame) bounds} for $\nu$. We call $\nu$ a {\it
Plancherel measure} if $A=B=1$. 
We say that a set $\Lambda$ in $\br^d$ is a {\it spectrum} for $\mu$ if the set $E(\Lambda)=\{e_\lambda : \lambda\in\Lambda\}$ is an orthonormal basis for $L^2(\mu)$.
\end{definition}

\begin{remark} \label{rm1.1}
Note that, as mentioned previously, if $\{ e_{\lambda_{n}}: n \in
\bz \} \in L^2(\mu)$ is a frame, then the measure $\nu = \sum_{n
\in \bz} \delta_{\lambda_{n}}$ is a frame measure.  Conversely,
if $\nu$ is purely atomic, i.e. $\nu = \sum_{n \in \bz} d_{n} \delta_{\lambda_{n}}$,
and is a frame measure for $\mu$, then $\{ \sqrt{d_{n}} e_{\lambda_{n}} \}$ is a
(weighted) Fourier frame for $L^2(\mu)$.  Indeed, we have
\begin{align*}
A \| f \|^{2}_{L^2(\mu)} \leq \|\widehat{f\,d\mu} \|_{L^2(\nu)}^{2} &= \sum_{n \in \bz} d_n |\widehat{f\,d\mu}(\lambda_{n}) |^2 \\
&= \sum_{n \in \bz} | \langle f, \sqrt{d_n} e_{\lambda_{n}} \rangle |^2 \leq B \| f\|_{L^2(\mu)}^{2}.
\end{align*}
Here we require that $\mu$ be a finite Borel measure in order for $e_{\lambda_{n}} \in L^2(\mu)$.  For the remainder of the paper, we shall assume that $\mu$ is a Borel probability measure, unless stated explicitly otherwise.
\end{remark}

\section{Qualitative Structure Results}

In this section we prove some general results concerning Bessel and frame measures $\nu$ for a given measure $\mu$. We begin with a proof that any Bessel measure $\nu$ must be {\it
locally finite}, i.e. $\nu(K)<\infty$ for all compact subsets $K$
of $\br^d$, and hence is $\sigma$-finite.  This will allow us the
use of the Fubini theorem.

\begin{proposition}\label{pr1.1}
If $\nu$ is a Bessel measure for the measure $\mu$ then there
exists a constant $C$ such that $\nu(K)\leq C\max\{1,\diam(K)^d\}$
for any compact subset $K$ of  $\br^d$.  Consequently, $\nu$ is
$\sigma$-finite.
\end{proposition}

\begin{proof}
Since $\widehat{d\mu}(0)=\mu(\br^d)>0$ there exits $\epsilon>0$
and $\delta>0$ such that $|\widehat{d\mu}(x)|^2\geq \delta$ for
$x\in B(0,\epsilon)$. Then for any $t\in \br^d$ we have
$$B=B\|e_t\|_{L^2(\mu)}^2\geq \int |\widehat{e_t\,d\mu}(x)|^2\,d\nu(x)=\int|\widehat{d\mu}(x-t)|^2\,d\nu(x)\geq \int_{B(t,\epsilon)}|\widehat{d\mu}(x-t)|^2\,d\nu(x)\geq \nu(B(t,\epsilon))\delta.$$
Therefore $\nu(B(t,\epsilon)) \leq B/\delta$. This implies that if a
compact set has diameter less than $\epsilon$, then its measure is
bounded by $B/\delta$. Since any compact set $K$ can be covered by
some universal constant times $\diam(K)^d$ such balls, the result
follows.

\end{proof}

\begin{theorem}\label{th1.1.1}
There exist finite compactly supported Borel measures that do not
admit frame measures.
\end{theorem}

\begin{proof}
Let $\mu=\chi_{[0,1]}\,dx+\delta_2$. Suppose $\nu$ is a frame
measure for $\mu$ with frame bounds $A,B>0$. Let
$f:=\chi_{\{2\}}$. Then $|\widehat{f\,d\mu}(t)|=1$ and
$\|f\|_{L^2(\mu)}=1$. The Bessel bound implies that $\nu(\br)\leq
B<\infty$. From this we obtain that for any $\epsilon>0$ there
exists $R>0$ such that $\nu(\br\setminus B_0(R))<\epsilon$.

Now take some $T$ large, arbitrary and let $g(x):=e^{-2\pi i
Tx}\chi_{[0,1]}$. We have
$$|\widehat{g\,d\mu}(t)|^2=\frac{\sin^2(\pi(T+t))}{\pi^2(T+t)^2}\quad(t\in\br).$$
Therefore $|\widehat{g\,d\mu}(t)|^2\leq 1$ for all $t\in\br$ and
taking $T\geq 2R$, we have
$$|\widehat{g\,d\mu}(t)|^2\leq \frac{1}{\pi^2(T-R)^2}\mbox{ for all }t\in(-R,R).$$
Then, using the lower frame bound
$$A=A\|g\|_{L^2(\mu)}^2\leq \int |\widehat{g\,d\mu}(t)|^2\,d\nu(t)=\int_{B(0,R)}|\widehat{g\,d\mu}(t)|^2\,d\nu(t)+\int_{\br\setminus B(0,R)}|\widehat{g\,d\mu}(t)|^2\,d\nu\leq \frac{1}{\pi^2(T-R)^2}\cdot \nu(\br)+\epsilon.$$
Letting $T\rightarrow\infty$ and $\epsilon\rightarrow0$ we obtain
that $A=0$, a contradiction.
\end{proof}

The next proposition shows that once a Bessel/frame measure is present, many others can be constructed.

\begin{proposition}\label{pr1.2}
For fixed positive constants $A$ and $B$, let $\B_B(\mu)$ denote the set of all Bessel measures with bound $B$ and let $\F_{A,B}(\mu)$ denote the set of all frame measures with bounds $A,B$.  Both sets, while possibly empty, are convex and closed under convolution with Borel probability measures.
\end{proposition}

\begin{proof}
To see that these sets are convex requires just a direct
computation. We check that $\B_B(\mu)$ is closed under convolution
with Borel probability measures, for the set of frame measures the
proof in analogous. Take $\nu\in \B_B(\mu)$ and let $\rho$ be a
Borel probability measure on $\br^d$. Take $f\in L^2(\mu)$. Then
$$\int |\widehat{f\,d\mu}(t)|^2\,d\nu*\rho(t)=\iint|\widehat{f\,d\mu}(t+s)|^2\,d\nu(t)\,d\rho(s)=\iint|\widehat{e_{-s}f\,d\mu}(t)|^2\,d\nu(t)\,d\rho(s)$$
$$\leq \int B\|e_{-s}f\|_{L^2(\mu)}^2\,d\rho(s)=B \int \|f\|_{L^2(\mu)}^2\,d\rho(s)=B\|f\|_{L^2(\mu)}^2.$$
\end{proof}

\begin{remark}
The set $\B_B(\mu)$ is never empty. One can take $\nu=\sum_{i\in I}c_i\delta_{\lambda_i}$ for some $\lambda_i\in\br^d$ and adjust the constants such that $\sum_{i\in I}c_i\leq B/\mu(\br^d)^2$. Then $\nu\in\B_B(\mu)$.
\end{remark}

\begin{corollary}\label{cor1.3}
If $\mu$ has a Bessel/frame measure $\nu$ then it has one which is
absolutely continuous with respect to the Lebesgue measure and
whose Radon-Nikodym derivative is $C^\infty$.
\end{corollary}

\begin{proof}
Convoluting $\nu$ with the Lebesgue measure on $[0,1]$ one obtains
a Bessel/frame measure which is absolutely continuous; then,
convoluting with $g\,dt$ where $g\geq 0$ is a compactly supported
$C^\infty$-function with $\int g(t)\,dt=1$, one obtains the
desired measure.
\end{proof}

The next theorem shows that the Bessel/frame property is preserved under approximations that use convolution kernels.

\begin{theorem}\label{th1.4}
Let $\lambda_n$ be an approximate identity, in the sense that
$\lambda_n$ is a Borel probability measure with the property that
$\sup\{\|t\| : t\in\supp{\lambda_n}\}\rightarrow 0$ as
$n\rightarrow\infty$. Suppose $\nu$ is a $\sigma$-finite Borel
measure, and suppose $\nu*\lambda_n$ are Bessel/frame
measures for $\mu$ with uniform bounds, independent of $n$. Then
$\nu$ is a Bessel/frame measure.
\end{theorem}

\begin{proof}
We show first that if $f$ is a continuous function on $\br^d$ then
for any $x\in\br^d$,
\begin{equation}
\int f(x+t)\,d\lambda_n(t)\rightarrow f(x)\mbox{ as }n\rightarrow
\infty. \label{eq1.4.1}
\end{equation}
Fix $\epsilon>0$. Since $f$ is continuous at $x$ there exits
$\delta>0$ such that $|f(x+t)-f(x)|<\epsilon$ if $\|t\|<\delta$.
There exists $n_\delta$ such that the support of $\lambda_n$ is
contained in $B(0,\delta)$ for all $n\geq n_\delta$. Then
$$\left|\int f(x+t)\,d\lambda_n(t)-f(x)\right|\leq\int|f(x+t)-f(x)|\,d\lambda_n(t)=\int_{B(0,\delta)}|f(x+t)-f(x)|\,d\lambda_n(t)\leq \int\epsilon\,d\lambda_n=\epsilon.$$

We prove first that $\nu$ is a Bessel measure with the same bound
$B$ as $\nu*\lambda_n$. Take $f\in L^2(\mu)$. Since
$\widehat{f\,d\mu}$ is continuous, we have by Fatou's lemma:
\begin{align*}
\int |\widehat{f\,d\mu}(x)|^2\,d\nu(x)&=\int\lim_n\int |\widehat{f\,d\mu}(x+t)|^2\,d\lambda_n(t)\,d\nu(x) \\
&\leq \liminf_n\iint |\widehat{f\,d\mu}(x+t)|^2\,d\lambda_n(t)\,d\nu(x) \\
&= \liminf_n\int |\widehat{f\,d\mu}(y)|^2\,d(\nu * \lambda_n)(y) \\
&\leq B\|f\|_{L^2(\mu)}^2.
\end{align*}

To obtain the lower bound we have to show that
\begin{equation}
\int |\widehat{f\,d\mu}(x)|^2\,d\nu*\lambda_n\rightarrow\int
|\widehat{f\,d\mu}(x)|^2\,d\nu. \label{eq1.4.2}
\end{equation}
For this, note first that given $M>0$,
$$\int_{\{\|x\|\leq M\}} \int |\widehat{f\,d\mu}(x+t)|^2\,d\lambda_n(t)\,d\nu(x)\rightarrow \int_{\{\|x\|\leq M\}}|\widehat{f\,d\mu(x)}|^2\,d\nu(x).$$
This follows from Lebesgue's dominated convergence theorem, since
$|\widehat{f\,d\mu}|$ is bounded, more precisely
$\|\widehat{f\,d\mu}\|_\infty\leq \|f\|_{L^1(\mu)}$, and since
$\nu(\{\|x\|\leq M\})<\infty$, by Proposition \ref{pr1.1}.

Secondly, given $\delta>0$ we will find $M>0$ such that for $n$
large enough

\begin{equation}
\int_{\{\|x\|\geq M\}} \int
|\widehat{f\,d\mu}(x+t)|^2\,d\lambda_n(t)\,d\nu(x)< \delta\mbox{
and }\int_{\{\|x\|\geq
M\}}|\widehat{f\,d\mu(x)}|^2\,d\nu(x)<\delta. \label{eq1.4.3}
\end{equation}

We have
$$\int_{\{\|x\|\geq M\}}|\widehat{f\,d\mu}(x+t)|^2\,d\nu(x)=\int\chi_{\{\|x\|\geq M\}}|\widehat{e_{-t}f\,d\mu}(x)|^2\,d\nu(x).$$

Since $\mu$ is compactly supported, for $t$ small enough, we have
that $e_{-t}f$ is close to $f$ in $L^2(\mu)$. Since $\nu$ is a
Bessel measure we obtain that $\widehat{e_{-t}f\,d\mu}$ is close
to $\widehat{f\,d\mu}$ in $L^2(\nu)$. Multiplying by the
characteristic function, we get that $\chi_{\{\|x\|\geq
M\}}\widehat{e_{-t}f\,d\mu}(x)$ is close to $\chi_{\{\|x\|\geq
M\}}\widehat{f\,d\mu}$ in $L^2(\nu)$. Since $\widehat{f\,d\mu}$
is in $L^2(\nu)$ we can find $M$ such that
$$\int\chi_{\{\|x\|\geq M\}}|\widehat{f\,d\mu}(x)|^2\,d\nu(x)<\delta/2.$$
Then, for $t$ small enough
\begin{equation} \label{eq1.4.4}
\int\chi_{\{\|x\|\geq M\}}|\widehat{e_{-t}f\,d\mu}(x)|^2\,d\nu(x)<\delta.
\end{equation}
Pick $n$ large enough so that inequality \eqref{eq1.4.4} holds for all $t$ in the support of $\lambda_n$,
then integrating the inequality \eqref{eq1.4.4} with respect to $\lambda_n$ we obtain
$$\int \int_{\{\|x\|\geq M\}} |\widehat{f\,d\mu}(x+t)|^2\,d\nu(x)\,d\lambda_n(t)< \delta.$$
The inequality \eqref{eq1.4.3} follows by an application of
Fubini's theorem.

The limit in \eqref{eq1.4.2} follows by splitting the integral in
two regions where $\|x\|\leq M$ or $\|x\|\geq M$. Then the lower
bound follows from \eqref{eq1.4.2}.
\end{proof}

In the following three proposition we present a general way of constructing Bessel/frame measures for a given measure. The idea is that if a convolution of two measures admits a Bessel measure then one can use the Fourier transform of one of the measures in the convolution as a weight for the Bessel measure to obtain a Bessel measure for the other measure in the convolution. The same is true for frame measures, but with certain restrictions.

\begin{proposition}\label{pr1.5}
Let $\mu$ , $\mu'$ be two Borel probability measures. For $f\in
L^1(\mu)$, the measure $(f\,d\mu)* \mu'$ is absolutely continuous
w.r.t. $\mu *\mu'$. We denote by $P_{\mu,\mu'}f$ or $Pf$ the
Radon-Nykodim derivative:
$$Pf=\frac{(f\,d\mu)*\mu'}{d(\mu*\mu')}.$$

\end{proposition}

\begin{proof}
Take a Borel set $E$ such that $\mu*\mu'(E)=0$. Then
$$0=\iint\chi_E(x+y)\,d\mu(x)\,d\mu'(y)=\int\mu(E-y)\,d\mu'(y).$$
This implies that $\mu(E-y)=0$ for $\mu'$-a.e. $y$.

Then $$\iint \chi_E(x+y)\,
f\,d\mu(x)\,d\mu'(y)=\iint\chi_E(x+y)f(x)\,d\mu(x)\,d\mu'(y)=\iint
\chi_{E-y}(x)f(x)\,d\mu(x)\,d\mu'(y)=0.$$ So
$(f\,d\mu)*\mu'(E)=0$.
\end{proof}

\begin{proposition}\label{pr04} Let $\mu$ and $\mu'$ be probability measures and
$1\leq p\leq \infty$.  If $f\in L^p(\mu)$ then the function $Pf$ is also
in $L^p(\mu*\mu')$ and
\begin{equation}
\|Pf\|_{L^p(\mu*\mu')}\leq\|f\|_{L^p(\mu)} \label{eq04.1}
\end{equation}
\end{proposition}

\begin{proof}
For $p=\infty$, let $g\in L^1(\mu*\mu')$. Then
\begin{multline*}
\left|\int gPf\,d\mu*\mu'\right|=\left|\int g\,d((f\,d\mu)*\mu')\right|=\left|\iint g(x+y)f(x)\,d\mu(x)\,d\mu'(y)\right| \\
\leq
\|f\|_\infty\left|\iint|g(x+y)|\,d\mu(x)\,d\mu(y)\right|=\|f\|_\infty\int
|g|\,d(\mu*\mu').
\end{multline*}

Suppose $\|Pf\|_\infty>\|f\|_\infty$. Then for
$\epsilon>0$ small, the set $A:=\{x:
|Pf(x)|\geq\|f\|_\infty+\epsilon\}$ has positive
$\mu*\mu'$-measure. Let $g:=\frac{|Pf|}{Pf}\chi_A$. We have that
$g\in L^1(\mu*\mu')$ and
\begin{equation*}
\left|\int g Pf\,d\mu*\mu'\right|=\int_A|Pf|\,d\mu*\mu'\geq (\|f\|_\infty+\epsilon)\mu*\mu'(A)=(\|f\|_\infty+\epsilon)\int |g|\,d\mu*\mu'.
\end{equation*}
This contradicts the previous computation.

For $p=1$, take $g(x)=|Pf(x)|/Pf(x)$ if $Pf(x)\neq 0$ and $g(x)=0$
otherwise. Then
\begin{multline*}
\int|Pf|\,d\mu*\mu'=\int gPf\,d\mu*\mu'=\int g\,d((f\,d\mu)*\mu') \\
=\iint g(x+y)f(x)\,d\mu(x)\,d\mu(y) \leq\iint |f(x)|\,d\mu(x)\,d\mu(y)=\|f\|_{L^1(\mu)}.
\end{multline*}

For $1<p<\infty$ the inequality follows from the Riesz-Thorin
interpolation theorem.

\end{proof}

\begin{proposition}\label{pr02}
Let $\mu$, $\mu'$ be probability measures. Assume that $\mu*\mu'$
has a Bessel measure $\nu$. Then $|\hat\mu'|^2\,d\nu$ is a Bessel
measure for $\mu$ with the same bound.

If in addition $\nu$ is a frame measure for $\mu*\mu'$ with bounds
$A$ and $B$, and $c\|f\|_{L^2(\mu)}^2\leq
\|Pf\|_{L^2(\mu*\mu')}^2$ for all $f\in L^2(\mu)$, then
$|\hat\mu'|^2\,d\nu$ is a frame measure for $\mu$ with bounds $cA$
and $B$.
\end{proposition}

\begin{proof}
Take $f\in L^2(\mu)$. Then

$$\int|\widehat{(f\,d\mu)}|^2\cdot|\widehat\mu'|^2\,d\nu=\int|\widehat{(f\,d\mu)*\mu'}|^2\,d\nu=\int|\widehat{(Pf\,d\mu*\mu')}|^2\,d\nu\geq A\|Pf\|_{L^2(\mu*\mu')}^2\geq cA\|f\|_{L^2(\mu)}^2.$$
The upper bound follows from Proposition \ref{pr04}.
\end{proof}

Next we prove some stability results.
\begin{theorem}\label{th1.9}
Let $\mu$ be a compactly supported Borel probability measure. If $\nu$ is a Bessel
measure for $\mu$ then for any $r>0$ there exists a constant $D>0$
such that
\begin{equation}
\int\sup_{|y|\leq r}|\widehat{f\,d\mu}(x+y)|^2\,d\nu(x)\leq
D\|f\|_{L^2(\mu)}^2,\mbox{ for all }f\in L^2(\mu). \label{eq1.9.1}
\end{equation}
If $\nu$ is a frame measure for $\mu$ then there exists constants
$\delta>0$ and $C>0$ such that
\begin{equation}
C\|f\|_{L^2(\mu)}^2\leq \int
\inf_{|y|\leq\delta}|\widehat{f\,d\mu}(x+y)|^2\,d\nu(x),\mbox{ for
all }f\in L^2(\mu). \label{eq1.9.2}
\end{equation}
\end{theorem}

\begin{proof}

Let $\epsilon:\br^d\rightarrow\br^d$ be some Borel measurable
function, such that $|\epsilon(x)|\leq r$ for all $x\in\br^d$.
Take $f\in L^2(\mu)$ and $x\in\br^d$. Let
$u:=\epsilon(x)/|\epsilon(x)|$ and
$$g_x(t):=\widehat{f\,d\mu}(x+tu),\quad(t\in\br).$$
The function $g$ is analytic and its derivatives are
$$g_x^{(k)}(t)=\int e^{-2\pi i (x+tu)\cdot y}(-2\pi i u\cdot y)^kf(y)\,d\mu(y)=((-2\pi i u\cdot y)^kf)^{\widehat{}}(x+tu).$$
Writing the Taylor expansion at $0$ we have
$$|g_x(|\epsilon(x)|)-g_x(0)|^2=\left|\sum_{k=1}^\infty \frac{g_x^{(k)}(0)}{k!}|\epsilon(x)|^k\right|^2$$
and using the Cauchy-Schwartz inequality
$$\leq \sum_{k=1}^\infty\frac{|g_x^{(k)}(0)|^2}{k!}\sum_{k=1}^\infty\frac{|\epsilon(x)|^{2k}}{k!}=\sum_{k=1}^\infty\frac{|g_x^{(k)}(0)|^2}{k!}\cdot(e^{|\epsilon(x)|^2}-1)\leq (e^{r^2}-1)\sum_{k=1}^\infty\frac{|g_x^{(k)}(0)|^2}{k!}$$
We use the Bessel bound and obtain
$$\int |g_x^{(k)}(0)|^2\,d\nu(x)=\int |((-2\pi i u\cdot y)^kf)^{\widehat{}}(x)|^2\,d\nu(x)\leq B\|(-2\pi i u\cdot y)^kf(y)\|_{L^2(\mu)}^2\leq (2\pi M)^{2k}B\|f\|_{L^2(\mu)}^2,$$
where the constant $M$ is chosen such that the support of $\mu$ is
contained in the ball $|x|\leq M$. Then
$$\int |\widehat{f\,d\mu}(x+\epsilon(x))-\widehat{f\,d\mu}(x)|^2\,d\nu(x)=\int |g_x(|\epsilon(x)|)-g_x(0)|^2\,d\nu(x)\leq (e^{r^2}-1)\int \sum_{k=1}^\infty\frac{|g_x^{(k)}(0)|^2}{k!}\,d\nu(x)$$$$=(e^{r^2}-1)\sum_{k=1}^\infty\frac{1}{k!}\int |g_x^{(k)}(0)|^2\,d\nu(x)\leq (e^{r^2}-1)\sum_{k=1}^\infty \frac{1}{k!}(2\pi M)^{2k}B\|f\|_{L^2(\mu)^2}^2=B(e^{r^2}-1)(e^{(2\pi M)^2}-1)\|f\|_{L^2(\mu)}^2.$$

Then, by Minkowski's inequality,
$$\left(\int |\widehat{f\,d\mu}(x+\epsilon(x))|^2\,d\nu(x)\right)^{\frac12}\leq \left(\int |\widehat{f\,d\mu}(x)|^2\,d\nu(x)\right)^{\frac12}+\left(\int |\widehat{f\,d\mu}(x+\epsilon(x))-\widehat{f\,d\mu}(x)|^2\,d\mu(x)\right)^{\frac12}$$
$$\leq (B^{\frac12}+(B(e^{r^2}-1)(e^{(2\pi M)^2}-1))^{\frac12})\|f\|_{L^2(\mu)}.$$
Taking $\epsilon(x)$ such that
$|\widehat{f\,d\mu}(x+\epsilon(x))|^2=\sup_{|y|\leq
r}|\widehat{f\,d\mu}(x+y)|^2$, we obtain \eqref{eq1.9.1}. Such a
function $\epsilon$ is Borel measurable since the function
$\widehat{f\,d\mu}$ is continuous, analytic in each variable.

For \eqref{eq1.9.2}, take $\delta=r>0$ small enough such that
$$C:=A^{\frac12}-(B(e^{r^2}-1)(e^{(2\pi M)^2}-1))^{\frac12}>0,$$
where $A$ is the lower frame bound for $\nu$.

Then, by Minkowski's inequality
$$\left(\int |\widehat{f\,d\mu}(x+\epsilon(x))|^2\,d\nu(x)\right)^{\frac12}\geq \left(\int |\widehat{f\,d\mu}(x)|^2\,d\nu(x)\right)^{\frac12}-\left(\int |\widehat{f\,d\mu}(x+\epsilon(x))-\widehat{f\,d\mu}(x)|^2\,d\mu(x)\right)^{\frac12}$$
$$\geq (A^{\frac12}-(B(e^{r^2}-1)(e^{(2\pi M)^2}-1))^{\frac12})\|f\|_{L^2(\mu)}.$$

Take $\epsilon(x)$ such that
$|\widehat{f\,d\mu}(x+\epsilon(x))|^2=\inf_{|y|\leq
r}|\widehat{f\,d\mu}(x+y)|^2$ and we obtain \eqref{eq1.9.2}.
\end{proof}

Using this stability of frame measures we can prove that a certain form of discretization will produce atomic frame measures from a general frame measure.
\begin{theorem}\label{th1.10}
If a measure $\mu$ has a Bessel/frame measure $\nu$ then it has
also an atomic one. More precisely, let $Q=[0,1)^d$ and $r>0$; if
$\nu$ is a Bessel measure for $\mu$ and $(x_k)_{k\in\bz^d}$ is a
set of points such that $x_k\in r(k+Q)$ for all $k\in\bz^d$ then
$$\nu':=\sum_{k\in\bz^d}\nu(r(k+Q))\delta_{x_k}$$
is a Bessel measure for $\mu$. We call $\nu'$ a discretization of the
measure $\nu$.

If $\nu$ is a frame measure for $\mu$ and $r$ is small enough then
the measure $\nu'$ defined above is a frame measure for $\mu$.
\end{theorem}

\begin{proof}
Define $\epsilon(x)=x_k-x$ if $x\in r(k+Q)$. Then,
$|\epsilon(x)|\leq r\sqrt{d}=:r'$ and for all $f\in L^2(\mu)$,
$$\int |\widehat{f\,d\mu}(x+\epsilon(x))|^2\,d\nu(x)=\sum_{k\in\bz^d}\int_{r(k+Q)}|\widehat{f\,d\mu}(x_k)|^2\,d\nu(x)=\sum_{k\in\bz^d}\nu(r(k+Q))|\widehat{f\,d\mu}(x_k)|^2.$$

But since
$$ \int \inf_{|y|\leq r'}|\widehat{f\,d\mu}(x+y)|^2\,d\nu(x)\leq \int |\widehat{f\,d\mu}(x+\epsilon(x))|^2\,d\nu(x)\leq \int \sup_{|y|\leq r'}|\widehat{f\,d\mu}(x+y)|^2\,d\nu(x),$$
everything follows from Theorem \ref{th1.9}.
\end{proof}

\begin{corollary}
If $\nu$ is a frame measure for $\mu$ and $r > 0$ is sufficiently small, then $\{ c_{k} e_{x_{k}} : k \in \bz^d \}$ is a weighted Fourier frame for $\mu$, where $x_{k} \in r(k + Q)$ and $c_{k} = \sqrt{\nu(r(k +Q))}$.
\end{corollary}

\begin{proof}
The measure $\nu' = \sum_{k \in \bz^d} c_{k}^2 \delta_{x_{k}}$ is a frame measure for $\mu$, so for any $f \in L^2(\mu)$ we have
\[
 A \| f \|^{2}_{\mu} \leq \int | \widehat{f \ d \mu} (t) |^2 d \nu' = \sum_{k \in \bz^d} c_{k}^2 | \widehat{f d \ \mu} (x_{k}) |^2  = \sum_{k \in \bz^d} |\langle f , c_{k} e_{x_{k}} \rangle|^2 \leq B \| f \|^{2}_{\mu}
\]
\end{proof}

\section{Beurling dimension}

In \cite{DHSW10}, Beurling dimension (of sequences) as defined in \cite{CKS08} is used to provide a partial characterization of Bessel sequences of exponentials.  We extend the definition of Beurling dimension to measures, and using this definition we obtain similar results concerning Bessel measures.  Specifically, we show that any Bessel measure for $\mu$ must have Beurling dimension which is sufficiently small, depending upon a certain property that $\mu$ may possess.  We begin with some basic properties, including certain invariances which will be useful later, of Beurling dimension of measures.

\begin{definition}\label{def2.1}
Let $Q$ be the unit cube $Q=[0,1)^d$. For a locally finite measure
$\nu$ and $\alpha\geq 0$ we define the {\it $\alpha$-upper
Beurling density} by
\begin{equation}
\D_\alpha(\nu):=\limsup_{R\rightarrow\infty}\sup_{x\in\br^d}\frac{\nu(x+RQ)}{R^\alpha}.
\label{eq2.1.1}
\end{equation}
We define the {\it (upper) Beurling dimension} of $\nu$ by
\begin{equation}
\dim\nu:=\sup\{\alpha\geq 0 : \D_\alpha(\nu)=\infty\}.
\label{eq2.1.2}
\end{equation}
\end{definition}

\begin{proposition}\label{pr2.2}
If $\nu$ is a locally finite Borel measure then
\begin{equation}
\D_\alpha(\nu)=\infty\mbox{ for }\alpha<\dim\nu\mbox{ and
}\D_\alpha(\nu)=0\mbox{ for }\alpha>\dim\nu. \label{eq2.2.1}
\end{equation}
In particular
\begin{equation}
\dim\nu=\inf\{\alpha\geq0 : \D_\alpha(\nu)=0\} \label{eq2.2.2}
\end{equation}
\end{proposition}

\begin{proof}
Take $\alpha<\dim\nu$. Then there exists $\alpha'>\alpha$ such
that $\D_{\alpha'}(\nu)=\infty$. Then
$$\D_\alpha(\nu)=\limsup_{R\rightarrow\infty}\sup_{x\in\br}\frac{\nu(x+RQ)}{R^\alpha}\geq \limsup_{R\rightarrow\infty}\sup_{x\in\br}\frac{\nu(x+RQ)}{R^{\alpha'}}=\infty.$$

If $\alpha>\dim\nu$ then take $\alpha'$ such that
$\dim\nu<\alpha'<\alpha$. From the definition of $\dim\nu$ we have
that $\D_{\alpha'}(\nu)<\infty$. Then
$$\D_\alpha(\nu)=\limsup_{R\rightarrow\infty}\sup_{x\in\br}\frac{\nu(x+RQ)}{R^\alpha}= \limsup_{R\rightarrow\infty}\sup_{x\in\br}\frac{\nu(x+RQ)}{R^{\alpha'}}\cdot\frac{R^{\alpha'}}{R^\alpha}$$$$\leq \limsup_{R\rightarrow\infty}\sup_{x\in\br}\frac{\nu(x+RQ)}{R^{\alpha'}}\cdot\limsup_{R\rightarrow\infty}\frac{R^{\alpha'}}{R^\alpha}=\D_{\alpha'}(\nu)\cdot0=0.$$
\end{proof}

\begin{proposition}\label{pr2.3}
The Beurling dimension can be computed by replacing the set $Q$ by
any set $O$ that is bounded and has an interior point.
\end{proposition}

\begin{proof}
There exist $a,b\in\br^d$ and $r_0,R_0>0$ such that
$$a+r_0O\subset Q\subset b+R_0O.$$
We have
$$\limsup_R\sup_x\frac{\nu(x+R(a+r_0O))}{R^\alpha}=\limsup_R\sup_y\frac{\nu(y+Rr_0O)}{R^\alpha}=\frac{1}{r_0^{-\alpha}}\limsup_{R'}\sup_y\frac{\nu(y+R'O)}{{R'}^\alpha}.$$
Therefore the two quantities involving $\limsup \sup\dots$ are
simultaneously 0 or $\infty$. Similarly
$$\limsup_R\sup_x\frac{\nu(x+R(b+R_0O))}{R^\alpha}=\frac{1}{R_0^{-\alpha}}\limsup_{R'}\sup_y\frac{\nu(y+R'O)}{{R'}^\alpha}.$$

Since we have
$$\limsup_R\sup_x\frac{\nu(x+R(a+r_0O))}{R^\alpha}\leq \limsup_R\sup_x\frac{\nu(x+RQ)}{R^\alpha}\leq \limsup_R\sup_x\frac{\nu(x+R(b+R_0O))}{R^\alpha}.$$
we see that the $\limsup\sup\dots$ involving $Q$ and the one
involving $O$ are at the same time $0$ or $\infty$. The result
then follows from Proposition \ref{pr2.2}.
\end{proof}

\begin{theorem}\label{th2.4}
Let $\nu$ be a locally finite measure and $\rho$ a Borel
probability measure. Then
$$\dim\nu*\rho=\dim\nu.$$
In other words, the Beurling dimension is invariant under
convolution with probability measures.
\end{theorem}

\begin{proof}
Let $\alpha_0:=\dim\nu$. Take $\alpha>\alpha_0$. Then
$\D_\alpha(\nu)=0$. Therefore, for any $\epsilon>0$, there exist
$R$ as large as we want such that
$$\frac{\nu(x+RQ)}{R^\alpha}<\epsilon\mbox{ for all }x\in\br^d.$$
Then, for all $x\in\br^d$,
$$\nu*\rho(x+RQ)=\int\nu(x+RQ-t)\,d\rho(t)\leq \int \epsilon R^\alpha\,d\rho(t)=\epsilon R^\alpha.$$
This implies that
$$\sup_{x\in\br^d}\frac{\nu*\rho(x+RQ)}{R^\alpha}\leq \epsilon,$$
and since $R$ can be taken arbitrarily large, we obtain that
$\D_\alpha(\nu*\rho)=0$ so $\dim\nu*\rho\leq\alpha_0$.

Now take $\alpha<\alpha_0$. We have $\D_\alpha(\nu)=\infty$. We use
Proposition \ref{pr2.3} with $O$ the unit ball. Given $M>0$ there
exist $R$ as large as we want such that
$$\sup_{x\in\br^d}\frac{\nu(x+RO)}{R^\alpha}>M.$$
This means that there exists $a\in\br^d$ such that
$$\nu(a+RO)\geq MR^\alpha.$$

Since $\rho(\br^d)=1$ we can pick $R$ large enough so that
$\rho(RO)\geq\frac12$. Then for any $t\in\br^d$ with $|t|<R$ we
have $a-t+2RO\supset a+RO$, and we have
$$\nu*\rho(a+2RO)=\int\nu(a-t+2RO)\,d\rho(t)\geq \int_{|t|<R}\nu(a-t+2RO)\,d\rho(t)$$$$\geq \int_{|t|<R}\nu(a+RO)\,d\rho(t)\geq MR^\alpha\rho(RO)\geq \frac{M}{2}R^\alpha.$$
Therefore
$$\sup_{x\in\br^d}\frac{\nu*\rho(x+2RO)}{(2R)^\alpha}\geq \frac{M}{2\cdot 2^\alpha}.$$
Since $M$ is arbitrary, this implies that
$\D_\alpha(\nu*\rho)=\infty$. Hence $\dim\nu*\rho\geq\alpha_0$,
and the result follows.
\end{proof}

\begin{theorem}\label{th2.5}
Let $\nu$ be a locally finite Borel measure and $r>0$. For each
$k\in\bz^d$, let $x_k$ be a point in $r(k+Q)$. Define the measure
$$\nu'=\sum_{k\in\bz^d}\nu(r(k+Q))\delta_{x_k}.$$
Then
$$\dim\nu'=\dim\nu.$$
In other words, the Beurling dimension is invariant under
discretization.
\end{theorem}

\begin{proof}
Take $\alpha>\dim\nu$. Then, given $\epsilon>0$, for $R$ large
enough
$$\frac{\nu(x+RQ)}{R^\alpha}<\epsilon\mbox{ for all }x\in \br^d.$$

Take $x\in\br^d$ arbitrary. We have
\begin{equation}
\nu'(x+RQ)=\sum_{k : x_k\in
x+RQ}\nu(r(k+Q))=\nu\left(\bigcup_{x_k\in x+RQ}r(k+Q)\right).
\label{eq2.5.11}
\end{equation}
Since $x_k\in (x+RQ)\cap(r(k+Q))$ we see that the union of cubes
$r(k+Q)$ of side $r$ in \eqref{eq2.5.11} intersecting the cube
$x+RQ$ of side $R$ is contained in a cube of side $R+2r$, which we
call $x'+(R+2r)Q$. Then
$$\frac{\nu'(x+RQ)}{R^\alpha}\leq \frac{\nu(x'+(R+2r)Q)}{R^\alpha}=\frac{\nu(x'+(R+2r)Q)}{(R+2r)^\alpha}\cdot\frac{(R+2r)^\alpha}{R^\alpha}<\epsilon\cdot 2$$
for $R$ large enough (independent of $x$). Then
$$\sup_{x\in\br^d}\frac{\nu'(x+RQ)}{R^\alpha}\leq 2\epsilon$$
for $R$ large, so $\dim\nu'\leq \alpha$ and so
$\dim\nu'\leq\dim\nu$.

Now take $\alpha<\dim\nu$. Given $M>0$ we can find $R$ as large as
we want and $x\in\br^d$ such that
$$\frac{\nu(x+RQ)}{R^\alpha}\geq M.$$

The cube $x+RQ$ is contained in the union $U$ of the cubes
$r(k+Q)$ that intersect it. This union $U$ is contained in some
cube of side $R+2r$, say $x'+(R+2r)Q$.

We have
$$\nu'(x'+(R+2r)Q)=\sum_{x_k\in x'+(R+2r)Q}\nu(r(k+Q))=\nu\left(\bigcup_{x_k\in x'+(R+2r)Q}r(k+Q)\right)\geq $$
$$\nu\left(\bigcup_{r(k+Q)\subset x'+(R+2r)Q}r(k+Q)\right)\geq \nu(U)\geq \nu(x+RQ).$$
Then
$$\frac{\nu'(x+(R+2r)Q)}{(R+2r)^\alpha}\geq \frac{\nu(x+RQ)}{R^\alpha}\cdot\frac{R^\alpha}{(R+2r)^\alpha}\geq \frac M2.$$
This proves that
$$\sup_{x\in\br^d}\frac{\nu'(x+(R+2r))}{(R+2r)^\alpha}\geq \frac M2,$$
so $\D_\alpha(\nu')=\infty$ so $\dim\nu'\geq\dim\nu$
\end{proof}

We now establish some upper bounds on the Beurling dimension of Bessel measures.  The following result is true for any Borel measure $\mu$; subsequent results (e.g. Theorems \ref{th2.6} and \ref{th3.1}) refine this basic result based on whether $\mu$ has additional structure.

\begin{theorem} \label{th2.2}
If $\nu$ is a Bessel measure for a Borel measure $\mu$ on $\br^d$, then $\dim \nu \leq d$.
\end{theorem}

\begin{proof}
We apply Proposition \ref{pr2.3} using the closed unit cube $\overline{Q}$ to estimate $\dim \nu$.  If $ \alpha > d$, we have
\begin{align*}
\D_{\alpha}(\nu) &= \limsup_{R \to \infty} \sup_{x \in \br^d} \frac{ \nu(x + R \overline{Q} ) }{R^{\alpha}} \\
&\leq \limsup_{R \to \infty} \frac{ C R^d  }{R^{\alpha}} \\
&= 0
\end{align*}
where the inequality follows from Proposition \ref{pr1.1}.
\end{proof}

\begin{definition}\label{def2.5}
We say that a Borel measure $\mu$ is {\it
ocasionally-$\alpha$-dimensional } if there exists a sequence of
Borel subsets $E_n$ and some constants $c_1,c_2>0$ such that
$\diam(E_n)$ decreases to $0$ as $n\rightarrow\infty$,
\begin{equation}
\sup_n\frac{\diam(E_n)}{\diam(E_{n+1})}<\infty \label{eq2.5.0}
\end{equation}
\begin{equation}
c_1\diam(E_n)^\alpha\leq \mu(E_n)\leq
c_2\diam(E_n)^\alpha,\quad(n\geq0). \label{eq2.5.1}
\end{equation}

\end{definition}

\begin{theorem}\label{th2.6}
Let $\mu$ be a occasionally-$\alpha$-dimensional measure and
suppose $\nu$ is a Bessel measure for $\mu$. Then
$\D_\alpha(\nu)<\infty$ and so $\dim(\nu)\leq \alpha$.
\end{theorem}

\begin{proof}
Let $E_n$ be a sequence of sets as in Definition \ref{def2.5}. Let
$\epsilon_n:=\diam(E_n)$ and pick $a_n\in E_n$. Choose $\delta>0$
such that
$$|e^{2\pi iy}-1|\leq \frac 12\mbox{ if }|y|\leq \delta.$$

Let $f_n:=\chi_{E_n}$. We have, for $|t|\leq
\frac{\delta}{\epsilon_n}$,
$$\left|\widehat{f_n\,d\mu}(t)-e^{2\pi i t\cdot a_n}\widehat{f_n\,d\mu}(0)\right|=\left|\int_{E_n}(e^{-2\pi it\cdot x}-e^{2\pi it\cdot a_n})\,d\mu(x) \right|
\leq \int_{E_n}|e^{2\pi it\cdot a_n}(e^{-2\pi i
t\cdot(x-a_n)}-1)|\,d\mu(x)$$$$\leq
\int_{E_n}\frac12\,d\mu=\frac12\mu(E_n),$$ since
$|x-a_n|<\diam(E_n)=\epsilon_n$ for $x\in E_n$ so
$|t\cdot(x-a_n)|\leq \frac{\delta}{\epsilon_n}\epsilon_n=\delta$.

Also
$$|e^{2\pi it\cdot a_n}\widehat{f_n\,d\mu}(0)|=\mu(E_n).$$
Then
$$|\widehat{f_n\,d\mu}(t)|\geq |e^{2\pi it\cdot a_n}\widehat{f_n\,d\mu}(0)|-\left|\widehat{f_n\,d\mu}(t)-e^{2\pi i t\cdot a_n}\widehat{f_n\,d\mu}(0)\right|\geq \frac12\mu(E_n)\geq\frac{c_1}{2}\epsilon_n^\alpha,$$
for all $|t|\leq \frac{\delta}{\epsilon_n}$.

Let $g_n:=\frac{f_n}{\mu(E_n)^{1/2}}$. Then
$\|g_n\|_{L^2(\mu_n)}=1$ and
$$|\widehat{g_n\,d\mu}(t)|^2\geq \frac{c_1^2}{4}\epsilon_n^{2\alpha}\cdot\frac{1}{\mu(E_n)}\geq \frac{c_1^2\epsilon_n^{2\alpha}}{4c_2\epsilon_n^\alpha}=:C\epsilon_n^\alpha,$$
for all $|t|\leq \frac{\delta}{\epsilon_n}$, and $C>0$.

Apply the Bessel inequality to the function $e^{2\pi i a\cdot
x}g_n(x)$: for all $a\in\br^d$:
$$B\geq \int |\widehat{e_ag_n\,d\mu}|^2\,d\nu\geq \int_{B(a,\frac{\delta}{\epsilon_n})}|\widehat{g_n\,d\mu}(t-a)|^2\,d\nu(t)\geq C\epsilon_n^\alpha\nu(B(a,\frac{\delta}{\epsilon_n})).$$
Then
$$\frac{1}{\delta^\alpha}\frac{B}{C}\geq \frac{\nu(B(a,\frac{\delta}{\epsilon_n}))}{\left(\frac{\delta}{\epsilon_n}\right)^\alpha}.$$

Now take $M\geq \epsilon_n/\epsilon_{n+1}$ for all $n$ (according
to \eqref{eq2.5.0}), and pick $R>0$ large. Let $n$ be such that
$\frac{\delta}{\epsilon_n}\leq R\leq
\frac{\delta}{\epsilon_{n+1}}$.

We have

$$\sup_{a\in\br^d}\frac{\nu(B(a,R))}{R^\alpha}\leq \sup_a\frac{\nu(B(a,\frac{\delta}{\epsilon_{n+1}}))}{\left(\frac{\delta}{\epsilon_n}\right)^\alpha}=
\sup_a\frac{\nu(B(a,\frac{\delta}{\epsilon_{n+1}}))}{\left(\frac{\delta}{\epsilon_{n+1}}\right)^\alpha}\cdot\frac{\epsilon_n^\alpha}{\epsilon_{n+1}^\alpha}\leq
\frac{1}{\delta^\alpha}\frac{B}{C} M^\alpha.$$

This shows that $\D_\alpha(\nu)<\infty$ so $\dim\nu\leq \alpha$.
\end{proof}

\begin{definition}\label{def2.7}
Let $\nu$ be a locally finite Borel measure on $\br^d$. The {\it
$d$-lower Beurling density} (or simply {\it the lower Beurling
density}) of $\nu$ is defined by
\begin{equation}
\D^-(\nu)=\liminf_{R\rightarrow\infty}\inf_{x\in\br^d}\frac{\nu(x+RQ)}{R^d}.
\label{eq2.7.1}
\end{equation}
\end{definition}

\begin{theorem}\label{th2.8}
Let $\mu$ be a compactly supported Borel probability measure on
$\br^d$. Suppose $\mu$ has a Bessel measure $\nu$ of positive
lower Beurling density. Then $\mu$ is absolutely continuous with
respect to the Lebesgue measure, and its Radon-Nikodym derivative
is in $L^2(\br^d)$.
\end{theorem}

\begin{proof}
Since $\nu$ has positive lower Beurling density, there exists
$R>0$ and $\delta>0$ such that
$$\frac{\nu(x+RQ)}{R^d}\geq \delta\mbox{ for all }x\in\br^d.$$
Then
$$\nu(x+RQ)\geq \delta R^d=:C>0,\mbox{ for all }x\in\br^d.$$

With Theorem \ref{th1.10}, the measure
$\nu'=\sum_{k\in\bz^d}\nu(R(k+Q))\delta_{Rk}$ is also a Bessel
measure for $\mu$. Note that $\nu(R(k+Q))\geq C$ for all
$k\in\bz$. Using Proposition \ref{pr1.2}, the convolution of
$\nu'$ with the probability measure
$\frac{1}{R^d}\chi_{RQ}(x)\,dx$, we obtain that the measure
$$\nu''=\frac{1}{R^d}\sum_{k\in\bz^d}\nu(R(k+Q))\chi_{Rk+RQ}(x)\,dx=:H(x)\,dx.$$
is a Bessel measure for $\mu$. But the previous remarks show that
$H(x)\geq \frac{1}{R^d}C$ for all $x\in\br^d$. Therefore we have
$$\int |\widehat\mu(x)|^2\,dx\leq \frac{R^d}{C}\int |\widehat\mu(x)|^2H(x)\,dx\leq \frac{R^d}{C}B,$$
where $B$ is the Bessel bound for the Bessel measure $H(x)\,dx$.

This means that $\hat\mu$ is in $L^2(\br^d)$. Then there exists
some function $g\in L^2(\br^d)$ whose Fourier transform $\hat g$
is $\hat\mu$.

Take $f$ an arbitrary $C^\infty$, compactly supported function on
$\br^d$. Then $\hat f$ is in $L^1(\br^d)\cap L^2(\br^d)$ and using
the Parseval relation (see e.g., \cite[Chapter VI.2]{Kat04}), we
have
$$\int f\,d\mu=\int \widehat f(\xi)\widehat\mu(-\xi)\,d\xi=\int \widehat f(\xi)\cj{\widehat\mu(\xi)}\,d\xi=\int\widehat f(\xi)\cj{\hat g(\xi)}\,d\xi=\int f(x)\cj{g(x)}\,dx.$$
Since $f$ is arbitrary, it follows that $d\mu=\cj g\,dx$. Note
that in particular this implies that $g$ is non-negative.
\end{proof}

For Lebesgue measure restricted to the unit interval, a necessary condition for a frame measure is that the Beurling dimension is 1.  We will prove this using a result in \cite{OSANN} concerning equivalent norms for the Paley-Wiener space.  Recall that the Paley Wiener space $PW_{\pi}$ is the collection of entire functions which are of exponential type $\pi$ and are square integrable on the real axis.  By the Paley-Wiener theorem \cite{PW34a}, this consists of all functions $f$ such that $f(z) = \int_{-\pi}^{\pi} g(t) e^{i t z} dt$ for some $g \in L^2([-\pi, \pi])$.  By reparametrizing, we consider $f(z) = \int_{-1/2}^{1/2} g(t) e^{-2 \pi i t z } dt$ for some $g \in L^2([-1/2,1/2])$.  Landau's inequality \cite{Lan67a} implies that if $\{ \lambda_{k} \}$ is a sequence such that there exists constants $0 < A,B < \infty$ such that for every $f \in PW_{\pi}$,
\begin{equation} \label{Eq:sampling}
 A \int_{\br} |f(t)|^2 dt \leq \sum_{k \in \bz} | f(\lambda_{k})|^2 \leq B \int_{\br} |f(t)|^2 dt
\end{equation}
then the Beurling density of $\{\lambda_{k}\}$ satisfies the inequalities
\begin{equation} \label{Eq:beurlingd}
1 \leq \mathcal{D}^{-} (\{ \lambda_{k} \} ) \leq \mathcal{D}^{+} ( \{ \lambda_{k} \} ) < \infty.
\end{equation}

If Equation \eqref{Eq:sampling} is satisfied, then $\{ \lambda_{k} \}$ is a sampling set for $PW_{\pi}$, and is equivalent to $\{ e_{\lambda_{k}} \} \subset L^2([-1/2, 1/2])$ being a Fourier frame.  From page 798 of \cite{OSANN}: for $r, \delta > 0$ and a measure $\nu$, we define the sequence
\[ \Lambda_{\nu}(r, \delta) = \{ kr : k \in \bz, \ \nu([kr, (k+1)r)) \geq \delta \}. \]

\begin{proposition}
A Borel measure $\nu$ yields an equivalent norm for $PW_{\pi}$, i.e., the measure $\nu$ is a frame measure for the Lebesgue measure on $[-1/2,1/2]$, if and only if the following hold:
\begin{enumerate}
\item There exists a positive constant $C$ such that $\nu([x, x+1)) \leq C$ for all $x \in \br$.
\item For all sufficiently small $r > 0$ there exists a $\delta = \delta(r) > 0$ such that $\Lambda_{\nu}(r, \delta)$ is a sampling set for $PW_{\pi}$.
\end{enumerate}
\end{proposition}

\begin{theorem}
If $\mu = \lambda|_{[-\frac{1}{2}, \frac{1}{2}]}$, and if $\nu$ is a frame measure for $\mu$, then $\mathcal{D}^{-}(\nu) > 0$, and hence $\dim(\nu) = 1$.
\end{theorem}

\begin{proof}
The measure $\nu$ is a frame measure for $\mu$ if and only if for every $f$ in the Paley-Wiener space $PW_{\pi}$ we have
\[ A \int |f(t)|^2 d \lambda(t) \leq \int |f(t)|^2 d \nu(t) \leq B \int |f(t)|^2 d \lambda(t), \]
and thus $\nu$ yields an equivalent norm for $PW_{\pi}$.  By the preceding Proposition we have that there exists an $r > 0$ and a $\delta > 0$ such that the sequence $\Lambda_{\nu}(\delta, r)$ is a sampling sequence for $PW$.

For $x,\ R \in \br$, we define the set
\[ W_{r}(x,R) = \{ kr : k \in \bz; \ \nu([kr, (k+1)r)) \geq \delta; \ [kr, (k+1)r) \subset x + RQ \}. \]
Note that $W_{r}(x,R) \subset \Lambda_{\nu}(\delta, r) \cap (x + RQ)$ and moreover the latter set is larger than the former by at most $1$ element.  Thus, we have for every $x\in\br$ and $R>0$
\[
\nu(x + RQ) \geq \sum_{kr \in W_{r}(x,R)} \nu([kr, (k+1)r)) \geq \delta \cdot \# W_{r}(x,R)
 \geq \delta \cdot \left(\# (\Lambda_{\nu}(\delta, r) \cap (x + RQ)) - 1 \right)
\]
Therefore, we have that
\begin{align*}
\mathcal{D}^{-}(\nu) &= \liminf_{R\rightarrow\infty}\inf_{x\in\br}\frac{\nu(x+RQ)}{R} \\
&\geq \delta \liminf_{R\rightarrow\infty}\inf_{x\in\br}\frac{\# (\Lambda_{\nu}(\delta, r) \cap (x + RQ)) - 1}{R} \\
&= \delta \mathcal{D}^{-}( (\Lambda_{\nu}(\delta, r) ) \\
&> 0
\end{align*}
by Equation (\ref{Eq:beurlingd}).  We have that $\mathcal{D}^{-}(\nu) \leq \mathcal{D}_{1}(\nu) < \infty$ by Theorem \ref{th2.2}.
\end{proof}

\section{Affine iterated function systems}

Our motivating examples, measures restricted to Cantor type sets, are invariant measures for iterated function systems.  For all such measures, we obtain an upper bound on the Beurling dimension of any Bessel measure.  For some such measures, as demonstrated in Theorems \ref{th3.4} and \ref{th3.6}, the structure of the iterated function system provides a way of constructing frame measures.  Specifically, the measure $\mu_{4}$ of \cite{JP98} is one such case; unfortunately, the measure $\mu_{3}$ for the usual Cantor set is not one of those cases.

\begin{definition}\label{def3.1}
Let $R$ be a $d\times d$ expanding integer matrix, i.e., all
eigenvalues $\lambda$ satisfy $|\lambda|>1$. Let $B$ be a finite
subset of $\bz^d$ of cardinality $\#B=:N$. We consider the
following {\it affine iterated function system}:
\begin{equation}
\tau_b(x)=R^{-1}(x+b),\quad(x\in\br^d,b\in B).
    \label{eq1_1}
\end{equation}

Since $R$ is expanding, the maps $\tau_b$ are contractions (in an
appropriate metric equivalent to the Euclidean one), and therefore
Hutchinson's theorem can be applied \cite{Hut81}: there exists a
unique compact set $X=X_B\subset\br^d$ such that
\begin{equation}
    X=\bigcup_{b\in B}\tau_b(X).
    \label{eq1_2}
\end{equation}

The set $X_B$ is called the {\it attractor} of the affine IFS.
    Moreover
\begin{equation}
    X_B=\left\{\sum_{k=1}^\infty R^{-k}b_k :  b_k\in B\mbox{ for all }k\in\bn\right\}.
    \label{eq1_3}
\end{equation}
There exists a unique Borel probability measure $\mu=\mu_B$ on
$\br^d$ such that
\begin{equation}
    \int f\,d\mu=\frac1N\sum_{b\in B}\int f\circ\tau_b\,d\mu,
    \label{eq1_4}
\end{equation}
for all compactly supported continuous functions $f$ on $\br^d$.
Moreover, the measure $\mu_B$ is supported on the set $X_B$. The
measure $\mu_B$ is called the {\it invariant measure} of the
affine IFS.

We say that the affine IFS has {\it no overlap} if
$\mu_B(\tau_b(X_B)\cap\tau_{b'}(X_B))=0$ for all $b\neq b'$ in
$B$.

\end{definition}

\begin{lemma}\label{lem4.3.1}
If the IFS has no overlap, and $E$ is a Borel set, then
\begin{equation}
\mu_B(\tau_b(E))=\frac{1}{N}\mu_B(E). \label{eq3.2.1}
\end{equation}
In particular
\begin{equation}
\mu_B(\tau_{b_1}\dots\tau_{b_n}(X_B))=\frac{1}{N^n},\quad(n\in\bn,b_1,\dots,b_n\in B)
\label{eq4.3.2}
\end{equation}
\end{lemma}
\begin{proof}
From the invariance equation we have that
$$\mu_B(E)=\frac{1}{N}\sum_{b\in B}\mu_B(\tau_b^{-1}(E))$$
so if $\mu_B(E)=0$ then $\mu_B(\tau_b^{-1}(E))=0$ for all $b\in
B$. Also
$$\mu_B(\tau_b(E))=\frac{1}{N}\sum_{b'\in B}\mu_B(\tau_{b'}^{-1}(\tau_b(E)))=\frac{1}{N}\mu_B(E)+\frac{1}{N}\sum_{b'\neq b}\mu_B(\tau_{b'}^{-1}(\tau_b(E))).$$

But using the no overlap and the fact that $\mu_B$ is supported on
$X_B$, we have
$$\mu_B(\tau_{b'}(\tau_b(E)))\leq \mu_B(\tau_{b'}^{-1}(\tau_b(X_B)))=\mu_B(X_B\cap\tau_{b'}^{-1}(\tau_b(X_B)))=\mu_B(\tau_{b'}^{-1}(\tau_{b'}(X_B)\cap\tau_b(X_B)))=0.$$
This proves \eqref{eq3.2.1}.

Applying \eqref{eq3.2.1}, we obtain by induction that
$$\mu_B(\tau_{b_1}\dots\tau_{b_n}(X_B))=\frac{1}{N^n}$$
for all $n\in\bn$ and $b_1,\dots,b_n\in B$.
\end{proof}

\begin{theorem} \label{th3.1}
In dimenson $d=1$, let $(R,B)$ be an affine IFS with no overlap, $N=\#B$, with $\mu_{B}$ the associated invariant measure.  If $\nu$ is a Bessel measure for $\mu_{B}$, the Beurling dimension of $\nu$ satisfies
\[ \dim \nu \leq \dfrac{\log N}{\log R}. \]
\end{theorem}

\begin{proof}
We prove that $\mu_{B}$ is occasionally-$\dfrac{\log N}{\log R}$-dimensional, whence the statement follows from Theorem \ref{th2.6}.  Let $b_0\in B$ and let $E_{n} = \tau_{b_0}^n(X_B)$; since $\diam(E_n) = R^{-n}\diam(E_0)$, we have that $\diam(E_n) \to 0$ and $\sup \dfrac{\diam(E_n)}{\diam(E_{n+1})} = R$. Moreover, by Equation (\eqref{eq4.3.2}), $\mu_{B}(E_n) = \dfrac{1}{N^n}$.  We have
\[ \diam(E_n)^{\log N / \log R} = (R^{-n})^{\log N / \log R}\diam(E_0)^{\log N/\log R} = N^{-n}\diam(E_0)^{\log N/\log R} \]
and therefore Equation (\ref{eq2.5.1}) is satisfied with $c_1 = c_2 = \diam(E_0)^{\log N/\log R}$.
\end{proof}

\begin{proposition}\label{pr3.2}
Let $(R,B)$ be an affine IFS with no overlap, $N=\#B$. On the
compact space $B^\bn$ consider the product probability measure
where each digit in $B$ has probability $1/N$. Define the encoding
map $\E_B:B^\bn\rightarrow X_B$,
$$\E_B(b_1b_2\dots)=R^{-1}b_1+R^{-2}b_2+\dots$$
Then $\E_B$ is onto, measure preserving and one-to-one on a set of
full measure.
\end{proposition}

\begin{proof}
Equation \eqref{eq4.3.2} proves that
$\E_B$ is measure preserving.  The fact that $\E_B$ is onto follows from \eqref{eq1_3}.

We check now that $\E_B$ is one-to-one on a set of full measure.
Let
$$E:=\bigcup_{b\neq b'}(\tau_b(X_B)\cap \tau_{b'}(X_B)),\quad F:=E\cup\bigcup_{n\geq 1, b_1\dots,b_n\in B}\tau_{b_1}\dots\tau_{b_n}(E).$$
From the computations above, we see that $F$ has measure zero. It
is also clear that $\tau_b^{-1}(F)\supset F$ for all $b\in B$.

Take $x\in X_B\setminus F$. Suppose $x=R^{-1}b_1+R^{-2}b_2+\dots$
(by \eqref{eq1_3}). Then $x=\tau_{b_1}(y)$ with
$y=R^{-1}b_2+R^{-2}b_3+\dots$. Since $x\not\in F\supset E$, $b_1$
is uniquely determined by $x$. We have
$y=\tau_{b_1}^{-1}(x)\not\in \tau_{b_1}^{-1}(F)$ so $y\not\in F$.
Repeating the argument, $b_2$ is uniquely determined and so on for
all $b_n$. Therefore $\E_B$ is one-to-one on the set of full
measure $\E_B^{-1}(X_B\setminus F)$.

\end{proof}

\begin{definition}\label{def3.3}
For two subsets $A,B$ of $\br^d$ we say that $A\oplus B=C$ if for
every element $c\in C$ there exist unique $a\in A$ and $b\in B$
such that $a+b=c$.
\end{definition}

\begin{theorem}\label{th3.4}
Let $(R,B)$ be an affine IFS with no overlap. Suppose there exists
a finite set $C$ such that $B\oplus C=:D$ is a complete set of
representatives for $\bz^d/R\bz^d$ and the affine IFS $(R,C)$ has
no overlap. Then $\mu_B$ has a Plancherel measure supported on a
lattice. More precisely, the measure $\mu_D$ has as spectrum a
lattice $\Gamma$, and the measure
$$\nu=\sum_{\gamma\in\Gamma}|\widehat{\mu_C}(\gamma)|^2\delta_\gamma$$
is a Plancherel measure for $\mu_B$. More generally, if $\nu$ is a
Bessel/frame measure for $\mu_D$, then $|\widehat\mu_C|^2\,d\nu$
is a Bessel/frame measure for $\mu_B$ with the same bounds.
\end{theorem}

\begin{proof}
The facts that $\mu_D$ is the Lebesgue measure on a set that tiles
$\br^d$ by some lattice $\Gamma$ and that there is no overlap for
$\mu_D$ are contained in \cite{CHR97}. We will use Proposition
\ref{pr02}.

Define the maps $\Phi:D^{\bn}\rightarrow B^\bn\times C^\bn$ and
$+:X_B\times X_C\rightarrow X_D$
$$\Phi(d_1d_2\dots)=(b_1b_2\dots,c_1c_2\dots),\mbox{ where }d_1=b_1+c_1, d_2=b_2+c_2,\dots,$$
$$+(x,y)=x+y.$$
The diagram below is commutative.

$$\begin{array}{ccccc}
D^{\bn}&\rightarrow&B^{\bn}&\times&C^{\bn}\\
\E_D\downarrow& &\E_B\downarrow& &\E_C\downarrow\\
X_D&\leftarrow& X_B&\times&X_C
\end{array}$$
It is easy to check that $\Phi$ is bijective and measure
preserving. From Proposition \ref{pr3.2} we conclude that the map
$+$ is bijective on sets of full measure, and measure preserving.
Since $+$ is measure preserving, we see that $\mu_B*\mu_C=\mu_D$
(alternatively, consider the Fourier transforms).

We can define the map $p:X_D\rightarrow X_B$
$$p(x+y)=x,\quad(x\in X_B,y\in X_C).$$

We will prove that for $f\in L^1(\mu_B)$ the function (see Proposition \ref{pr1.5} for the definition)
$Pf=P_{\mu_B,\mu_C}f=f\circ p$. Indeed, for a bounded Borel
function $g$ we have
\begin{multline*}
\int g\,d((f\,d\mu_B)*\mu_C)=\iint g(x+y)f(x)\,d\mu_B(x)\,d\mu_C(y)= \\
\iint g(x+y)f\circ p(x+y)\,d\mu_B(x)\,d\mu_C(y)=\int gf\circ p\,d(\mu_B*\mu_C).
\end{multline*}

In addition, we have
$$\int f\circ p\,d(\mu_B*\mu_C)=\iint f(p(x+y))\,d\mu_B(x)\,d\mu_C(y)=\iint f(x)\,d\mu_B(x)\,d\mu_C(y)=\int f\,d\mu_B.$$
Hence
$$\|Pf\|_{L^2(\mu_B*\mu_C)}=\|f\|_{L^2(\mu_B)},\quad(f\in L^2(\mu_B)).$$
Everything follows from Proposition \ref{pr02}.

\end{proof}

\begin{corollary}\label{cor3.5}
Let $\mu_4$ be the invariant measure for the affine IFS with $R=4$
and $B=\{0,2\}$ and let $\mu_4'$ be the invariant measure for the
affine IFS with $R=4$ and $C=\{0,1\}$. Then
$|\widehat{\mu_4'}(x)|^2\,dx$ and
$\sum_{n\in\bz}|\widehat{\mu_4'}(n)|^2\delta_n$ are Plancherel
measures for $\mu_4$.  Therefore, $\{ |\widehat{\mu_4'}(n)| e_{n} \} \subset
L^2(\mu_{4})$ is a Parseval frame.

Moreover, if $\{ e_{\lambda_{n}} \} \subset L^2[0,1]$ is a Fourier frame, then
$\{ |\widehat{\mu_4'}(\lambda_n)| e_{\lambda_n} \} \subset
L^2(\mu_{4})$ is a frame, with the same frame bounds.
\end{corollary}

\begin{proof}
With $B=\{0,2\}$ and $C=\{0,1\}$ we have $B\oplus
C=\{0,1,2,3\}=D$, and the measures $\mu_B,\mu_C$ have no overlap.
The measure $\mu_D$ is the Lebesgue measure on $[0,1]$. Then
everything follows from Theorem \ref{th3.4}.
\end{proof}

In \cite{DHSW10} it is shown that a Bessel sequence of exponentials for $\mu_{4}$ must have Beurling dimension at most $\dfrac{1}{2}$.  We point out here that the sequence
\[ \mathcal{Y} := \{ n : |\widehat{\mu_4'}(n)|^2 \neq 0 \} \]
has Beurling dimension of $1$.  Indeed, we can determine the zero set of $\widehat{\mu_4'}(t)$ as is done in \cite{JP98}:
\[ \widehat{\mu_4'}(t) = \prod_{n=0}^{\infty} \dfrac{1}{2}( 1 + e^{2 \pi i t/4^{n}} ) = e^{2 \pi i t/3} \cdot \prod_{n=1}^{\infty} \cos( \dfrac{\pi t}{4^{n}} ). \]
From this we see that the zero set is
\[ \mathcal{Z}\left( \widehat{\mu_4'}(t) \right) = \{ 4^n( 4 \mathbb{Z} + 2 ) : n = 0, 1, \dots \}, \]
which has Beurling density $\mathcal{D}_{1} \left( \mathcal{Z}\left( \widehat{\mu_4'}(t) \right) \right) = \dfrac{1}{3}$.  Therefore, the sequence of exponentials
\[ \{ |\widehat{\mu_4'}(n)| e_{n} : n \in \mathcal{Y} \}, \]
which forms a Parseval frame in $L^2(\widehat{\mu}_{4})$ by Corollary \ref{cor3.5}, has Beurling density of $\dfrac{2}{3}$, and hence Beurling dimension of $1$.

However, when we consider this sequence as a measure $\nu = \sum_{n\in\bz}|\widehat{\mu_4'}(n)|^2\delta_n$, we must have by Theorem \ref{th2.6} that $\nu$ has Beurling dimension at most $\dfrac{1}{2}$.

\begin{theorem} \label{th3.6}
Suppose that $B,C,D$ are as in Theorem \ref{th3.4}.  For functions $f \in L^2(\mu_{B})$ such that $Pf := P_{\mu_{B},\mu_{C}} f$ is continuous on the support of $\mu_{D}$ and its Fourier transform is in $L^1(\mathbb{R}^d)$, we have the following Fourier reconstruction theorem:  for every $t$ in the support of $\mu_{B}$,
\[ f(t) = \int \left( \widehat{f d \mu_{B}}(x) \cdot \widehat{\mu_{C}}(x) \right) e^{2 \pi i t \cdot x} dx. \]
\end{theorem}

\begin{proof}
Since $\mu_{D}$ is Lebesgue measure restricted to a set, we have by the Fourier inversion theorem that for $t$ in the support of $\mu_{D}$,
\begin{align*}
Pf (t) &= \int \widehat{Pf} (x) e^{2 \pi i x \cdot t} dx \\
&= \int \left\{ \int Pf (y) e^{-2 \pi i y \cdot x} d \mu_{D}(y) \right\}  e^{2 \pi i x \cdot t} dx \\
&= \int \left\{ \int e^{-2 \pi i y \cdot x} Pf (y) d (\mu_{B}*\mu_{C})(y)  \right\}  e^{2 \pi i x \cdot t} dx \\
&= \int \left\{ \int e^{-2 \pi i y \cdot x} d ((f\, d\mu_{B})*\mu_{C})(y)  \right\}  e^{2 \pi i x \cdot t} dx \\
&= \int ((f\,d\mu_{B})*\mu_{C})\sphat \ (x) e^{2 \pi i x \cdot t} dx \\
&= \int \left( \widehat{f\,d\mu_{B}} (x) \cdot \widehat{\mu_{C}}(x) \right) e^{2 \pi i x \cdot t} dx.
\end{align*}
However, for $t$ in the support of $\mu_{B}$, $Pf(t) = f(t)$ as in the proof of Theorem \ref{th3.4}.
\end{proof}

\bibliographystyle{alpha}
\bibliography{eframes}

\begin{thebibliography}{DHSW11}

\bibitem[CHR97]{CHR97}
J.-P. Conze, L.~Herv{\'e}, and A.~Raugi.
\newblock Pavages auto-affines, op\'erateurs de transfert et crit\`eres de
  r\'eseau dans {${\bf R}^d$}.
\newblock {\em Bol. Soc. Brasil. Mat. (N.S.)}, 28(1):1--42, 1997.

\bibitem[CKS08]{CKS08}
Wojciech Czaja, Gitta Kutyniok, and Darrin Speegle.
\newblock Beurling dimension of {G}abor pseudoframes for affine subspaces.
\newblock {\em J. Fourier Anal. Appl.}, 14(4):514--537, 2008.

\bibitem[DHS09]{DHS09}
Dorin~Ervin Dutkay, Deguang Han, and Qiyu Sun.
\newblock On the spectra of a {C}antor measure.
\newblock {\em Adv. Math.}, 221(1):251--276, 2009.

\bibitem[DHSW11]{DHSW10}
Dorin~Ervin Dutkay, Deguang Han, Qiyu Sun, and Eric Weber.
\newblock On the {B}eurling dimension of exponential frames.
\newblock {\em Adv. Math.}, 226:285--297, 2011.

\bibitem[DHW11]{DHW11a}
Dorin~Ervin Dutkay, Deguang Han, and Eric Weber.
\newblock Bessel sequences of exponentials on fractal measures.
\newblock {\em J. Functional Anal.}, 261(9):2529--2539, 2011.

\bibitem[DS52]{DS52a}
R.~Duffin and A.~Schaeffer.
\newblock A class of nonharmonic {F}ourier series.
\newblock {\em Trans. Amer. Math. Soc.}, 72:341--366, 1952.

\bibitem[GH03]{GH03a}
J.-P. Gabardo and D.~Han.
\newblock Frames associated with measurable spaces.
\newblock {\em Adv. Comput. Math.}, 18(2-4):127--147, 2003.
\newblock Frames.

\bibitem[HL08]{HL08}
Tian-You Hu and Ka-Sing Lau.
\newblock Spectral property of the {B}ernoulli convolutions.
\newblock {\em Adv. Math.}, 219(2):554--567, 2008.

\bibitem[Hut81]{Hut81}
John~E. Hutchinson.
\newblock Fractals and self-similarity.
\newblock {\em Indiana Univ. Math. J.}, 30(5):713--747, 1981.

\bibitem[IP00]{MR1744572}
Alex Iosevich and Steen Pedersen.
\newblock How large are the spectral gaps?
\newblock {\em Pacific J. Math.}, 192(2):307--314, 2000.

\bibitem[JKS07]{MR2338387}
Palle E.~T. Jorgensen, Keri~A. Kornelson, and Karen~L. Shuman.
\newblock Affine systems: asymptotics at infinity for fractal measures.
\newblock {\em Acta Appl. Math.}, 98(3):181--222, 2007.

\bibitem[JP98]{JP98}
Palle E.~T. Jorgensen and Steen Pedersen.
\newblock Dense analytic subspaces in fractal {$L\sp 2$}-spaces.
\newblock {\em J. Anal. Math.}, 75:185--228, 1998.

\bibitem[JP99]{MR1700084}
Palle E.~T. Jorgensen and Steen Pedersen.
\newblock Spectral pairs in {C}artesian coordinates.
\newblock {\em J. Fourier Anal. Appl.}, 5(4):285--302, 1999.

\bibitem[Kat04]{Kat04}
Yitzhak Katznelson.
\newblock {\em An introduction to harmonic analysis}.
\newblock Cambridge Mathematical Library. Cambridge University Press,
  Cambridge, third edition, 2004.

\bibitem[Lan67]{Lan67a}
H.~J. Landau.
\newblock Necessary density conditions for sampling and interpolation of
  certain entire functions.
\newblock {\em Acta Math.}, 117:37--52, 1967.

\bibitem[Li07]{MR2297038}
Jian-Lin Li.
\newblock {$\mu\sb {M,D}$}-orthogonality and compatible pair.
\newblock {\em J. Funct. Anal.}, 244(2):628--638, 2007.

\bibitem[LS02]{LS02a}
Yurii~I. Lyubarskii and Kristian Seip.
\newblock Weighted {P}aley-{W}iener spaces.
\newblock {\em J. Amer. Math. Soc.}, 15(4):979--1006 (electronic), 2002.

\bibitem[{\L}W06]{MR2200934}
Izabella {\L}aba and Yang Wang.
\newblock Some properties of spectral measures.
\newblock {\em Appl. Comput. Harmon. Anal.}, 20(1):149--157, 2006.

\bibitem[MZ09]{MZ09a}
U.~Molter and L.~Zuberman.
\newblock A fractal {P}lancherel theorem.
\newblock {\em Real Anal. Exchange}, 34(1):69--85, 2009.

\bibitem[OCS02]{OSANN}
Joaquim Ortega-Cerd{\`a} and Kristian Seip.
\newblock Fourier frames.
\newblock {\em Ann. of Math. (2)}, 155(3):789--806, 2002.

\bibitem[PW87]{PW34a}
R.~Paley and N.~Wiener.
\newblock {\em Fourier transforms in the complex domain}, volume~19 of {\em
  American Mathematical Society Colloquium Publications}.
\newblock American Mathematical Society, Providence, RI, 1987.
\newblock Reprint of the 1934 original.

\bibitem[Str90]{MR1078738}
Robert~S. Strichartz.
\newblock Self-similar measures and their {F}ourier transforms. {I}.
\newblock {\em Indiana Univ. Math. J.}, 39(3):797--817, 1990.

\bibitem[Str93a]{MR1081941}
Robert~S. Strichartz.
\newblock Self-similar measures and their {F}ourier transforms. {II}.
\newblock {\em Trans. Amer. Math. Soc.}, 336(1):335--361, 1993.

\bibitem[Str93b]{MR1237052}
Robert~S. Strichartz.
\newblock Self-similar measures and their {F}ourier transforms. {III}.
\newblock {\em Indiana Univ. Math. J.}, 42(2):367--411, 1993.

\bibitem[Str00]{MR1785282}
Robert~S. Strichartz.
\newblock Mock {F}ourier series and transforms associated with certain {C}antor
  measures.
\newblock {\em J. Anal. Math.}, 81:209--238, 2000.

\bibitem[Str06]{MR2279556}
Robert~S. Strichartz.
\newblock Convergence of mock {F}ourier series.
\newblock {\em J. Anal. Math.}, 99:333--353, 2006.

\bibitem[Yua08]{MR2443273}
Yan-Bo Yuan.
\newblock Analysis of {$\mu\sb {R,D}$}-orthogonality in affine iterated
  function systems.
\newblock {\em Acta Appl. Math.}, 104(2):151--159, 2008.

\end{thebibliography}

\end{document}